\documentclass[12pt]{amsart}
\usepackage{amsfonts}
\usepackage{amssymb}
\usepackage{amsmath}

\usepackage{hyperref}

\newtheorem{theorem}{Theorem}[section]
\newtheorem{lemma}[theorem]{Lemma}
\newtheorem{proposition}[theorem]{Proposition}
\newtheorem{corollary}[theorem]{Corollary}
\theoremstyle{definition}
\newtheorem{definition}[theorem]{Definition}
\newtheorem{example}[theorem]{Example}

\theoremstyle{remark}
\newtheorem{remark}[theorem]{Remark}
\numberwithin{equation}{section}

\begin{document}
	\title
	{On Mittag-Leffler $d$-orthogonal polynomials}
	\author{Abdessadek Saib}
	\address{Department of Mathematics, Tebessa University, Tebessa 12022, Algeria}
	\email{\href{mailto: A. Saib}{sad.saib@gmail.com}}
	
	\keywords{d-orthogonal polynomials; Hahn's property; linear combination; difference equation}

	\begin{abstract}
		This paper presents a first result of a long term research project dealing with the construction of $d$-orthogonal polynomials with Hahn's property. We shall show that the latter class could be characterized by expanding a polynomial as a finite sum of first  derivatives of the elements of the sequence and we shall explain how this characterization could be used to construct Hahn-classical d-orthogonal polynomials as well. 
		In this paper we look for solutions of linear combinations of the first derivatives of two consecutive elements of the sequence by considering the derivative operator and  Delta (discrete) operator. The resulting polynomials constitute a particular class of Laguerre $d$-orthogonal polynomials and a generalization of Mittag-Leffler polynomials, respectively.
		
	\end{abstract}
	
	\maketitle
	
\section{Introduction}

We aim to start a construction of $d$-orthogonal polynomials of Hahn type, that is to say, the $d$-orthogonal polynomials with $d$-orthogonal derivatives (towards Askey tableaux). 
Our idea is based on the fact that both of the latter sequences are $d$-orthogonal. This means that we can express any polynomial from the sequence as linear combination in terms the derivative sequence's elements. Moreover, the linear combination should be finite to guarantee the $d$-orthogonality of the derivative sequence. In the present paper we shall look at this type of polynomials as solutions of a linear combination by considering the first two consecutive terms from the linear combination in below and this will be our starting paper 
\begin{equation}
P_{n}(x)=\sum_{\nu=0}^{d+1}\lambda _{n,\nu}Q_{n-\nu}(x),\ \forall n\geq
0,  \label{LC0}
\end{equation}
where $Q_{n}=\left( n+1\right) ^{-1}{\mathcal{L}}P_{n+1}$, $n\geq 0$, and $\mathcal{L}$ is a lowering operator, that is, a linear operator that decreases in one unit the degree of a polynomial and such that ${\mathcal{L}}(1)=0$. The aim of this paper is to consider the case $\mathcal{L}:=\Delta_w$.

According to Hahn property, orthogonal polynomials of Hahn type (classical when d=1) are referred to as the orthogonal polynomials with orthogonal derivatives. 
As mentioned above, in this paper we shall look at the solutions of (\ref{LC0}) when $\lambda_{n,\nu}=0$ for $2\leq \nu\leq d+1$. The resulting polynomials have the following generating functions 
\begin{equation}
G\left(x,t\right) =
\exp\left\{\frac{xt}{1-at} +b_{0}+b_{1}t+...+b_{d-1}t^{d-1}\right\}, \label{SZ0}
\end{equation}

\begin{equation}
K\left( x,t\right) =\left( \frac{1-\beta t}{1-\alpha t}\right) ^{x/(\alpha-\beta) }
\exp \left\{ \sum_{i=0}^{d-1}b_{i}t^{i}\right\} ,
\label{SZ1}
\end{equation}
corresponding to the case $\mathcal{L}:=d/dx$ and $\mathcal{L}:=\Delta_w$, respectively.

We would like to mention on one hand that the $d$-orthogonal polynomials generated by (\ref{SZ0}) are of Laguerre type \cite[p. 10]{CheikhZag2}.  On the other hand, we will call the polynomials generated by (\ref{SZ1}) the Mittag-Leffler $d$-orthogonal polynomials since they reduce to classical Mittag-Leffler orthogonal polynomials with $d=1$ and $\alpha=-\beta=1$. Furthermore, if $\alpha$ or $\beta$ is zero this generating function yields Charlier $d$-orthogonal polynomials  studied by Ben Cheikh and Zaghouani in \cite{CheikhZag1}. Furthermore, it will be of interest to consider the same problem with different operators such as $D_q$, $D_{q,w}$ and $D_{p,q}$ and, for instance, to look at $d$-orthogonal polynomials at the quadratic lattice.

The main new results of the manuscript are presented in Sections \ref{sec:const} and \ref{sec:dual}. 
In Section \ref{sec:qo} we present basic concepts related with the $d$-quasi-orthogonality. First, we present a characterization of the definition introduced by Maroni occupied with an example of Laguerre $d$-orthogonal polynomials. Then  we shall show that there is a gap in the definition of Maroni which leads us to distinguish between the $d$-quasi-orthogonality of order exactly $l$ (which agrees with Maroni's definition) and at most $l$. 
The  first part of Section \ref{sec:const} contains an explicit expression of the first structure relation corresponding to the first two consecutive terms of the second structure relation. Then, we specify the latter structure relation by considering the case $\mathcal{L}=d/dx$ in the second part \ref{sec:Lag}. The resulting family of polynomials constitutes a subclass of Laguerre $d$-orthogonal polynomials.  The third part \ref{sec:ML} deals with the case $\mathcal{L}=\Delta_w$. In this case, the exponential generating function constitutes a generalization of the generating function of Mittag-Leffler polynomials. Interesting properties of this family, structure relations as well as difference equations are presented. In the last Section \ref{sec:dual} we use the quasi-monomiality to determine the dual sequence only for the discrete case (for the Laguerre case it suffices to replace $\Delta_w$ by $D:=d/dx$ and repeat the same process or just take the limit as $w$ goes to zero).

\section{Quasi-orthogonality and linear combinations}\label{sec:qo}

The generalized rising factorial is defined by $(x|w ) _{0}=1$ and 
\begin{equation*}
( x|w ) _{n}=x( x+w )( x+2w )...( x+(n-1)w ), \ \ n\geq 1,
\end{equation*}
and generalized falling factorial is defined by 
$\left\langle x|w \right\rangle _0=1$ and 
\begin{equation*}
\left\langle x|w \right\rangle_n=x( x-w )( x-2w )...( x-(n-1)w ) , \ \ n\geq 1.
\end{equation*}

Let us remark that the rising factorial and the falling factorial are connected through $( x|w ) _n=\left\langle x+(n-1)w|w \right\rangle_n$ and $\left\langle x|w \right\rangle_n=( x-(n-1)w|w ) _{n}$.
Notice also that when $w=1$, the rising factorial reduces to Pochhammer symbol, i.e. 
$( x|1 ) _{n}:=(x)_n=x( x+1 )( x+2 )...( x+n-1 )$.

The generating function of the generalized falling factorials can be obtained directly from the binomial series as follows
\begin{equation}
\left( 1+w t\right) ^{x/w }=\sum\limits_{n=0}^{\infty }\frac{t^{n}}{n!}\left\langle x|w \right\rangle _{n}.\label{GFFF}
\end{equation}

Now, let $\left\{ P_{n}\right\} _{n\geq 0}$ be a sequence of monic polynomials with $\deg P_{n}=n$, $n\geq 0$. The dual sequence $\left\{ u_{n}\right\} _{n\geq 0}$, $u_{n}\in \mathcal{P}'$, of $\left\{ P_{n}\right\} _{n\geq 0}$ is defined by duality bracket
denoted throughout  as $\left\langle u_{n},P_{m}\right\rangle
:=\delta _{n,m},$ $n,m\geq 0$. The latter equality can be regarded as a  bi-orthogonality between two sequences. 

Before we dive into the d-orthogonality, let us briefly  recall the standard orthogonality.  
A sequence $\left\{ P_{n}\right\} _{n\geq 0}$ is said to be orthogonal  with respect to a linear functional $u$ in the linear space of polynomials with complex coefficients, if 
\begin{equation*}
\left\langle u,P_{m}P_{n}\right\rangle :=r_{n}\delta _{n,m},\quad n,m\geq
0,\quad r_{n}\neq 0,\ n\geq 0.
\end{equation*}
In this case, necessarily $u$ is proportional to $u_0$, 
i.e. $u=\lambda u_{0}$, $\lambda \neq 0$. 

For a generalization of the above standard orthogonality we will deal with
the concept of $d$-orthogonality. Let us recall the definition and some
characterizations which will be needed in the sequel. Throughout this work
all the sequences of polynomials are supposed to be monic.

\begin{definition}
	\label{T4}\cite{Maronid} A sequence of monic polynomials 
	$\left\{P_{n}\right\} _{n\geq 0}$ is said to be a $d$-orthogonal polynomial
	sequence, in short a $d$-OPS, with respect to the $d$-dimensional vector of
	linear forms $\mathcal{U}=\left( u_{0},...,u_{d-1}\right) ^{T}$ if 
	\begin{equation}
	\left\{ 
	\begin{array}{l}
	\left\langle u_{r},x^{m}P_{n}\left( x\right) \right\rangle =0,\quad n\geq
	md+r+1,\quad m\geq 0,\vspace{0.2cm} \\ 
	\left\langle u_{r},x^{m}P_{md+r}\left( x\right) \right\rangle \neq 0,\quad
	m\geq 0,
	\end{array}
	\right.  \label{B5}
	\end{equation}
	for each $0\leq r\leq d-1.$
\end{definition}

The first and second conditions of (\ref{B5}) are called,  respectively, the
d-orthogonality conditions and the d-regularity conditions. In this case,
the d-dimensional vector form $\mathcal{U}$ is called regular. Notice further that
if $d=1$, then we meet again the notion of usual (standard) orthogonality.

The following characterization constitutes an analogue of Favard's theorem.

\begin{theorem}
	\label{BT2}\cite{Maronid} Let $\left\{ P_{n}\right\} _{n\geq 0}$ be a monic
	sequence of polynomials, then the following statements are equivalent.
	
	\begin{enumerate}
		\item[ (a)] The sequence $\left\{ P_{n}\right\} _{n\geq 0}$ is $d$-OPS with
		respect to $\mathcal{U}=\left( u_{0},...,u_{d-1}\right)$.
		
		\item[ (b)] The sequence $\left\{ P_{n}\right\} _{n\geq 0}$ satisfies a $%
		(d+2)$-term recurrence relation 
		\begin{equation}
		P_{m+d+1}\left( x\right) =\left( x-\beta _{m+d}\right) P_{m+d}\left(
		x\right) -\sum\nolimits_{\nu =0}^{d-1}\gamma _{m+d-\nu
		}^{d-1-\nu}P_{m+d-1-\nu }\left( x\right) ,\ \ m\geq 0,  \label{B6}
		\end{equation}
		with the initial data 
		\begin{equation}
		\left\{ 
		\begin{array}{l}
		P_{0}\left( x\right) =1,\ \ P_{1}\left( x\right) =x-\beta _{0},\vspace{0.15cm%
		} \\ 
		P_{m}\left( x\right) =\left( x-\beta _{m-1}\right) P_{m-1}\left( x\right)
		-\sum\nolimits_{\nu =0}^{m-2}\gamma _{m-1-\nu }^{d-1-\nu}P_{m-2-\nu }\left(
		x\right) ,\ \ 2\leq m\leq d,%
		\end{array}
		\right.  \label{B7}
		\end{equation}
		and the regularity conditions $\gamma _{m+1}^0\neq 0$, $m\geq 0$.
		
	\end{enumerate}
\end{theorem}

Now we recall the concept of quasi-orthogonality and some characterizations.

\begin{definition}
	\label{ST10}\cite{Maronid} A sequence $\left\{ P_{n}\right\} _{n\geq 0}$ is
	said to be  $d$-quasi-orthogonal of order $s$ with respect to the form $\mathcal{U}%
		=\left( u_{0},...,u_{d-1}\right) ^{T}$, if for every $0\leq r\leq d-1$,
		there exist $s_{r}\geq 0$ and $\sigma _{r}\geq s_{r}$ integers such that 
		\begin{equation}
		\left\{ 
		\begin{array}{l}
		\left\langle u_{r},P_{m}P_{n}\right\rangle =0,\ \ n\geq \left(
		m+s_{r}\right) d+r+1,\ \ m\geq 0,\vspace{0.2cm} \\ 
		\left\langle u_{r},P_{\sigma _{r}}P_{\left( \sigma _{r}+s_{r}\right)
			d+r}\right\rangle \neq 0,\ \ m\geq 0,%
		\end{array}%
		\right.  \label{M12}
		\end{equation}%
			with $s=\underset{0\leq r\leq d-1}{\max }s_{r}$.
\end{definition}

If the linear form $\mathcal{U}$ in the definition above is regular, then there exists another sequence of polynomials, say $\left\{ Q_{n}\right\} _{n\geq 0}$, $d$-orthogonal with respect to $\mathcal{U}$. The question to think about now is: what is the connection between these two sequences? Maroni in his papers doesn't provide any information on the latter connection. Unfortunately, in our work we shall use this hard stage. Next, we recall two characterizations of the $d$-quasi-orthogonality rely only on the polynomials. 
First, we have the following 

\begin{proposition}\label{Prop1}\cite{Sadek2,Saib5}
	Suppose that $\left\{ P_{n}\right\}_{n\geq 0}$ is $d$-OPS with respect to $\mathcal{U}$. 
	Then a sequence of polynomials $\left\{ Q_{n}\right\}_{n\geq 0}$ is strictly $d$-quasi-orthogonal of order $l$ with respect to $\mathcal{U}$ if and only if the following relation holds
	\begin{equation}
	Q_{n}\left( x\right) =\sum\limits_{i=n-dl}^{n}a_{n,i}P_{i}\left( x\right) ,\ \ n\geq dl,
	\label{ESS}
	\end{equation}%
	with $a_{n,n-dl}\neq 0.$
\end{proposition}

We shall give a motivating example of the latter proposition. For this end, we want to give a generalization of the hypergeometric polynomials discussed in \cite{JordaanQuasiOrth}. Indeed, in that paper, the authors proved the following lemma which is given explicitly here
\begin{lemma}
	Let $n\in \mathbb{N}$, $k\in\left\{1,2,...,n-1\right\}$ and $\alpha_2,..,\alpha_p,\beta_1,...,\beta_q \in \mathbb{R}$ such that $\alpha_2,..,\alpha_p$,$\beta_1,...,\beta_q \notin \left\{0,-1,-2,...,-n\right\}$ and 
	$\alpha_2\notin \left\{0,1,2,...,k-1\right\}$. Then
	\begin{align}
	&_pF_q\left( \left. 
	\begin{tabular}{c}
	$-n, \ \alpha_2+1 , ... ,  \ \alpha_p\vspace{0.15cm}$ \\ 
	$\beta_1 , ... , \beta_q $
	\end{tabular}
	\right\vert z\right)\label{LinerHyeprgFunct}	\\\displaystyle 
	&=\sum_{i=0}^{k}(-1)^i\binom{k}{i} \frac{\left\langle n\right\rangle _{i}\left\langle n+\alpha_2-i\right\rangle _{k-i}}{\left\langle \alpha_2\right\rangle _k}\ _pF_q\left( \left. 
	\begin{tabular}{c}
	$-n+i, \ \alpha_2-k+1 , ... ,  \ \alpha_p\vspace{0.15cm}$ \\ 
	$\beta_1 , ... , \beta_q $%
	\end{tabular}%
	\right\vert z\right).\nonumber
	\end{align}
\end{lemma} 

Accordingly, we could give a linear combination of some hypergeometric type d-OPS. 
\begin{example}\label{Ex1}
	The Laguerre d-OPS denoted $L_n^{(\alpha_1,\dots,\alpha_d)}(x)$ are defined in terms of the generalized hypergeometric function $_1F_d$ as follows
	\begin{equation}
	\begin{array}{lr}
	L_n^{\overrightarrow{\alpha_d}}(x):=L_n^{(\alpha_1 ,\dots, \alpha_d)}(x)=\ _1F_d\left( \left. 
	\begin{tabular}{c}
	$-n \vspace{0.15cm}$ \\ 
	$\alpha_1+1 ,\dots,\alpha_d+1  $
	\end{tabular}%
	\right\vert x\right), \\ \alpha_i\neq -1,-2,\dots, \ \ i=1,\dots,d. 
	\end{array}\label{d-Laguerre1}	
	\end{equation}
	By taking $p=1$, $q=d$ and $\beta_i=\alpha_i+1$ with $\alpha_i\neq
	-1,-2,...$ in (\ref{LinerHyeprgFunct}), we readily get the following representation
	\begin{eqnarray}\nonumber
	\hspace*{-.5cm}\sum_{k=0}^{dl}(-1)^k\binom{dl}{k} \frac{\left\langle  n\right\rangle _{k}
		\left(\beta+dl+1\right)_{n-k}}
	{\left(\beta+1\right)_{n}}\ 
	L_{n-k}^{(\alpha_1 , ... , \alpha_d)}(x)\\=\ _2F_{d+1}\left( \left. 
	\begin{tabular}{c}
	$-n, \ \beta+dl+1 \vspace{0.15cm}$ \\ 
	$\alpha_1+1 , ... , \alpha_d+1 ,\ \beta+1 $%
	\end{tabular}%
	\right\vert x\right):=P_n^{(\alpha_1 , ... , \alpha_d,\beta)}(x)
	\label{LinerCombLaguer}
	\end{eqnarray}
	with  $\beta\neq -1,-2,...$

The latter linear combination (\ref{LinerCombLaguer}) together with Proposition \ref{Prop1} show that the polynomial $P_n^{\overrightarrow{\alpha_d},\beta}(x)$ is $d$-quasi-orthogonal of order $l$ with respect to $P_n^{\overrightarrow{\alpha_d}}(x)$.
	
Notice also that if there exists $i$, say $i=1$ such that $\alpha_1=\beta+dl$, then the expansion (\ref{LinerCombLaguer}) reduces to 
	\begin{equation}
	\hspace*{-.5cm}\sum_{k=0}^{l}(-1)^k\binom{l}{k} \frac{\left\langle n\right\rangle _{k}\left(\alpha_1+1\right)_{n-k}}{\left(\beta+1\right)_{n}}\ 
	L_{n-k}^{(\alpha_1 , ... , \alpha_d}(x)=L_n^{(\beta,\alpha_2 , ... , \alpha_d)}(x).
	\label{LinerCombLaguer2}
	\end{equation}
Remark that (\ref{LinerCombLaguer2}) is a connection formula between two sequences of Laguerre $d$-OPS such that they differ in the first parameter,	which shows, in turn the possibility of expanding the Laguerre $d$-OPS in terms of Laguerre $d$-OPS.
	The latter expansion shows further that a linear combination of d-OPS could be again d-OPS \cite{Sadek4}. It can be used also to construct semi-classical $d$-orthogonal polynomials examples \cite{Sadek3} (of hypergeometric type). 
\end{example}

It is worthy to notice, as you can see in the definition \ref{ST10}, Maroni defines the $d$-quasi-orthogonality as $\left\langle u_{t},P_{n}\right\rangle =0$ for $n> dl+t $, 
the question now is what happens for values between $d(l-1)+1$ and $dl-1$?
Indeed, characterizations of the above definition are given by only considering the case $n>dl+t$ which is equivalent to assume the quasi-orthogonality of order exactly $l$. Next, we shall distinguish between the $d$-quasi-orthogonality of order exactly $l$ and at most $l$. 
\begin{definition}\cite{Saib5}\label{NDQO1}
	A sequence $\left\{ P_{n}\right\} _{n\geq 0}$ is
	$d$-quasi-orthogonal of order at most $l$ with respect to the form $\mathcal{U}=\left( u_{0},...,u_{d-1}\right) ^{T}$, if there exists an integer $1\leq r \leq d$ such that for every $0\leq t \leq d-1$, there exist integer numbers $l_{t}\geq 0$ and $s _{t}\geq l_{t }$ such that 
	\begin{equation}
	\left\{ 
	\begin{array}{l}
	\left\langle u_{t},P_{m}P_{n}\right\rangle =0,\ \ 
	n\geq (m+l_{t}-1) d+r+t+1,\ \ m\geq 0,\vspace{0.2cm} \\ 
	\left\langle u_{t},P_{s _{t}}P_{\left(s		_{t}+l_{t}-1\right)d+r+t}\right\rangle \neq 0,\ \ m\geq 0.
	\end{array}%
	\right.  \label{NQO2}
	\end{equation}
\end{definition}

The latter helps us to extract as well as to close the implication between the first and the second structure relation (interested reader on quasi-orthogonality and Hahn's property could look at \cite{Saib5} for more details). For instance, in this paper we shall be interested to the following characterizations
\begin{proposition}
	Let $\left\{ P_{n}\right\}_{n\geq 0} $ be a d-OPS with respect to $\mathcal{U}$. The
	following properties are equivalent:
	
\begin{enumerate}
	\item[(i)] $\left\{ P_{n}\right\}_{n\geq 0} $ is a sequence of Hahn-classical $d$-orthogonal polynomials.

\item[(ii)] There exist complex numbers $\lambda _{n,\nu}$ not all zero, such that \cite{Sadek4,Saib5} 
		\begin{equation}
		P_{n}(x)=\sum_{i=0}^{d+1}\lambda _{n,i}Q_{n-i}(x),\ \forall n\geq
		0.  \label{NCC}
		\end{equation}
	\end{enumerate}
\end{proposition}

Now, we would like to mention that the connection (\ref{NCC}) could be used to enumerate all the $\mathcal{L}$-classical $d$-OPS of Hahn type. Therefore, we shall focus next, on the linear combination (\ref{NCC}) and we shall look for its solutions by considering, in this paper, a  linear combination of the first two consecutive terms. Hence if we assume that $d=1$, then with an appropriate choice of the operator we should obtain some families of Askey scheme.

\section{Constructing OPS classical in the Hahn sense}\label{sec:const}

To start with, let us remark from (\ref{NCC}) that if $\lambda_{n,\nu}=0$ for $1\leq \nu\leq d+1$, then the solutions of these equations are $\mathcal{L}$-Appell $d$-orthogonal polynomials (see for instance \cite{CheikhZag1,DouakAppell,Zaghouani} the respective cases of $\mathcal{L}$).

Now suppose that $\lambda_{n,\nu}=0$ for $2\leq \nu\leq d+1$, i.e., 
\begin{equation}
P_{n}(x) = Q_{n}(x)- \lambda_{n} Q_{n-1}(x), \ \ \ \lambda_{n}:=\lambda_{n,1}   \label{NCCd}
\end{equation}
In order to determine all classical $d$-OPS satisfying (\ref{NCCd}), we shall explicitly determine the corresponding generating functions. For this end, it is more convenient to transform (\ref{NCCd}) to certain initial value problem. 

First of all, besides structure relations (\ref{NCC}) we have one more interesting structure relation inspired by \cite{Sadek4}

\begin{proposition}\label{SR}
	For $d\geq 2$ the polynomials generated by (\ref{NCCd}) satisfy the following structure relation 
	\begin{equation}
	(x-c)Q_{n-1}(x)=P_{n}(x)+(\lambda_{n}+\xi_{n-1}-c)P_{n-1}(x)-\sum_{i=2}^{d}\sum_{j=i}^{d}
	\frac{\eta_{n-j}^{d-j}}{\lambda_{n-i}...\lambda_{n-j}}P_{n-i}(x). \label{SR2}
	\end{equation}
\end{proposition}

To prove the latter proposition we need the following lemma based on the paper's results \cite{Sadek4}
\begin{lemma}\label{lem4.2}
	The recurrence coefficients of the two sequences of polynomials generated by (\ref{NCCd}) satisfy the following
	\begin{align}
	\lambda_{n-1}&\left[\frac{\eta_{n-d+1}^{0}}{\lambda_{n}\lambda_{n-1}...\lambda_{n-d+1}}+
	\frac{\eta_{n-d+2}^{1}}{\lambda_{n}\lambda_{n-1}...\lambda_{n-d+2}}+...+
	\frac{\eta_{n-2}^{d-3}}{\lambda_{n}\lambda_{n-1}\lambda_{n-2}}+
	\frac{\gamma_{n-1}^{d-3}}{\lambda_{n}\lambda_{n-1}}\right]\nonumber
	\\&
	=\frac{\eta_{n-d}^{0}}{\lambda_{n-2}\lambda_{n-3}...\lambda_{n-d}}+
	\frac{\eta_{n-d+1}^{1}}{\lambda_{n-2}\lambda_{n-3}...\lambda_{n-d+1}}+...+
	\frac{\eta_{n-2}^{d-2}}{\lambda_{n-2}}.
	\label{RRPQ}
	\end{align}
\end{lemma}
\begin{proof}
	If $\left\{ P_{n}\right\} _{n\geq 0}$ and $\left\{ Q_{n}\right\} _{n\geq 0}$ 
	are two $d$-OPS connected through (\ref{NCCd}), then \cite[eq. (25)]{Sadek4}
	\begin{equation}
	\begin{array}{rl}
	\lambda_{n+d}\eta _{n}^{0} & =\lambda_{n}\gamma _{n+1}^{0},\ \ n\geq 1,\vspace{0.2cm} \\
	\eta _{n}^{k} & =\gamma _{n}^{k}+\lambda_{n+d-k-1}\eta _{n}^{k+1}-\lambda_{n}\gamma
	_{n+1}^{k+1},\ \ 0\leq k\leq d-2,\ \ n\geq 1.
	\end{array} \label{E30}
	\end{equation}
	Now use the second equality in (\ref{E30}) to replace $f(\eta _{n}^{k+1},\gamma_{n+1}^{k+1})$ by $g(\eta _{n}^{k},\gamma _{n}^{k})$. Substitute recursively the latter fact in the left hand side of (\ref{RRPQ}) to get its right hand side. 
\end{proof}

\begin{proof}[Proof of Proposition \ref{SR}]
	We have from \cite[eq.(31)]{Sadek4} using also the first equality of (28) that
	\begin{align*}
	\beta_{n}&=c-\lambda_n-\frac{\eta_{n}^{d-1}}{\lambda_{n}}-\frac{\eta_{n-1}^{d-2}}{\lambda_{n}\lambda_{n-1}}-...-\frac{\eta_{n-d+2}^{1}}{\lambda_{n}\lambda_{n-1}...\lambda_{n-d+2}}
	-\frac{\eta_{n-d+1}^{0}}{\lambda_{n}\lambda_{n-1}...\lambda_{n-d+1}}	
	\\
	\frac{\gamma_{n}^{d-1}}{\lambda_n}&=c-\lambda_n-\xi_{n-1}-\frac{\eta_{n-1}^{d-2}}{\lambda_{n}\lambda_{n-1}}-...-\frac{\eta_{n-d+2}^{1}}{\lambda_{n}\lambda_{n-1}...\lambda_{n-d+2}}-
	\frac{\eta_{n-d+1}^{0}}{\lambda_{n}\lambda_{n-1}...\lambda_{n-d+1}}.\label{RRPQ1}	
	\end{align*}
	Inserting the latter expressions in the recurrence relation of $P_n(x)$ written as follows 
	\begin{align*}
	\gamma_n^{d-1}P_{n-1}=(x-\beta_{n})P_n
	-\sum\nolimits_{\nu =0}^{d-1}\gamma _{n-\nu}^{d-1-\nu}P_{n-1-\nu}-P_{n+1}
	\end{align*}
	and then replace each term of $P_k$ using (\ref{NCCd}) we get
	\begin{align*}
	\left[c-\lambda_n-\xi_{n-1}-\frac{\eta_{n-1}^{d-2}}{\lambda_{n}\lambda_{n-1}}...
	-\frac{\eta_{n-d+1}^{0}}{\lambda_{n}\lambda_{n-1}...\lambda_{n-d+1}}\right]P_{n-1}
	=P_n\\
	+\left[c-x-\frac{\eta_{n-1}^{d-2}}{\lambda_{n}\lambda_{n-1}}-...
	-\frac{\eta_{n-d+1}^{0}}{\lambda_{n}\lambda_{n-1}...\lambda_{n-d+1}}\right]Q_{n-1}
	\\
	+\left(\eta_{n-1}^{d-1}-\frac{\lambda_{n-1}}{\lambda_{n}}\gamma_{n}^{d-1}\right)Q_{n-2}
	+\sum\nolimits_{i =2}^{d}\eta _{n-i}^{d-i}Q_{n-1-i},
	\end{align*}
	which can be written using again (\ref{NCCd}), lemma \ref{lem4.2} and (\ref{E30}) as follows
	\begin{equation*}
	\begin{array}{l}
	\left(c-\lambda_n-\xi_{n-1}\right)P_{n-1} =P_n+\left(c-x\right)Q_{n-1}\vspace*{2mm}
	\\\displaystyle
	-\left[\frac{\eta_{n-d}^{0}}{\lambda_{n-2}...\lambda_{n-d}}+
	\frac{\eta_{n-d+1}^{1}}{\lambda_{n-2}...\lambda_{n-d+1}}+...+
	\frac{\eta_{n-2}^{d-2}}{\lambda_{n-2}}\right]Q_{n-2}
	+\sum\nolimits_{i =2}^{d}\eta _{n-i}^{d-i}Q_{n-1-i},
	\end{array}
	\end{equation*}
	and also in the following expression
	\begin{equation*}
	\begin{array}{l}
	\left(c-\lambda_n-\xi_{n-1}\right)P_{n-1} =P_n+\left(c-x\right)Q_{n-1}\vspace*{2mm}
	\\\displaystyle
	-\frac{\eta_{n-2}^{d-2}}{\lambda_{n-2}}\left(Q_{n-2}-\lambda_{n-2}Q_{n-3}\right)
	-\frac{\eta_{n-2}^{d-2}}{\lambda_{n-2}\lambda_{n-3}}
	\left(Q_{n-2}-\lambda_{n-2}\lambda_{n-3}Q_{n-4}\right)
	\\\displaystyle
	-...-\frac{\eta_{n-d}^{0}}{\lambda_{n-2}\lambda_{n-3}...\lambda_{n-d}}
	\left(Q_{n-2}-\lambda_{n-2}\lambda_{n-3}...\lambda_{n-d}Q_{n-d-1}\right).
	\end{array}
	\end{equation*}
	Thus, on account of (\ref{NCCd}) with some rearrangement we get the structure relation  (\ref{SR2}). 
\end{proof}

\begin{remark}
Let us remark that if we multiply  (\ref{SR2}) by $\lambda_{n}$ and replace $\lambda_{n}Q_{n-1}(x)$ using (\ref{NCCd}), then the structure relation (\ref{SR2}) can also be written as 
\begin{equation*}
\begin{array}{cl}
(x-c)Q_{n}(x)=(x+\lambda_{n}-c)P_{n}(x)+\lambda_{n}(\lambda_{n}+\xi_{n-1}-c)P_{n-1}(x)
\\\displaystyle -\lambda_{n}\sum_{i=2}^{d}\sum_{j=i}^{d}
\frac{\eta_{n-j}^{d-j}}{\lambda_{n-i}...\lambda_{n-j}}P_{n-i}(x). \label{SR2-2}
\end{array}
\end{equation*}
\end{remark}


\subsection{Differential operator: Laguerre type polynomials}\label{sec:Lag}

Let us begin with $\mathcal{L}=\frac{d}{dx}$. This case has been given as an example for the regularity of linear combination of $d$-orthogonal polynomials in \cite{Sadek4}. First, let us remark that $\exp \left\{ xt/\left( 1-at\right)\right\}$ is the unique solution of the following parametric first order differential equation $(1-at)y'(x)=ty(x)$ with $y(0)=1$.

Next, we shall denote by $R(x,t)$ the exponential generating function corresponding to the sequence of polynomials generated by (\ref{NCCd}). Then, it is straightforward to transform (\ref{NCCd}) to the following initial value problem
\begin{equation}
t R(x,t) = (1-a t)\frac{\partial}{\partial x} R(x,t).\label{Sys2}
\end{equation}
Therefore, according to the previous paragraph, the general solution of (\ref{Sys2}) takes the following form 
\begin{equation}
R(x,t)=A(t)\exp \left\{ xt/\left( 1-at\right) \right\}, \hspace{5mm} A(0)=1. \label{GAH1}
\end{equation}

Now, by taking the first derivative with respect to the variable $t$ and set \begin{equation*}
A'(t)/A(t)=\sum\nolimits_{n\geq 0} \alpha_k \frac{t^k}{k!},
\end{equation*} 
we get
\begin{equation*}\begin{array}{rl}
P_{n+1}(x)&=\left(x+2an+\alpha_0\right)P_n(x)-n\left[a^2(n-1)+2a\alpha_0-\alpha_1\right]P_{n-1}(x)
\vspace*{1.5mm}\\&+\sum\limits_{k=2}^n\binom{n}{k}\left(\alpha_k-2ak\alpha_{k-1}+a^2k(k-1)\alpha_{k-2}\right)P_{n-k}(x).
\end{array}
\end{equation*}
On the other hand, since the sequence of polynomials $\left\{ P_{n}\right\} _{n\geq 0}$ is $d$-orthogonal, then we should have $\alpha_{k}=0$ for $k\geq n-2$. Accordingly, we obtain 
\begin{equation}
A(t)=\exp \left\{ \sum\nolimits_{k=0}^{d-1}b_{k}\ \frac{t^{k}}{k!} \right\} .
\end{equation}
The expression of $A(t)$ shows that the generating function (\ref{GAH1}) is a subclass of the exponential generating function of Laguerre type d-OPS \cite[p.10]{CheikhZag2}. Then, it is more convenient to consider the following generating function
\begin{align}
G\left(x,t\right)& =\sum\limits_{n= 0}^{\infty}P_n(x)\frac{t^n}{n!}\nonumber\\&
=\left(1-at\right)^\beta \exp\left\{\frac{xt+\theta}{1-at} +b_{0}+b_{1}t+...+b_{d-1}\frac{t^{d-1}}{(d-1)!}\right\}
.\label{NLag}
\end{align}

In this case, assuming that $b_i\equiv0$ if $i\geq d$, we have
\begin{equation*}
\begin{array}{cl}
P_{n+1}(x)&=\left( x+a(\theta-\beta+2n)+b_{1}\right) P_{n}(x)
\vspace{2mm}\\&-n\left[a^2(n-\beta-1)+2ab_1-b_2\right]P_{n-1}(x)\\
&+\displaystyle\sum\nolimits_{i=2}^{d}
\left[ \frac{b_{i+1}}{i!}-2a\frac{b_{i}}{(i-1)!}+a^2\frac{b_{i-1}}{(i-2)!}\right]
\left\langle  n\right\rangle  _{i}P_{n-i}(x),
\end{array}
\end{equation*}
and 
\begin{equation*}
\begin{array}{cl}
Q_{n+1}(x)&=\left( x+a(\theta-\beta+2n+1)+b_{1}\right)Q_{n}(x)
\vspace{2mm}\\&-n\left[a^2(n-\beta)+2ab_1-b_2\right]Q_{n-1}(x)\\
&+\displaystyle\sum\nolimits_{i=2}^{d}
\left[ \frac{b_{i+1}}{i!}-2a\frac{b_{i}}{(i-1)!}+a^2\frac{b_{i-1}}{(i-2)!}\right]
\left\langle  n\right\rangle _{i}Q_{n-i}(x).
\end{array}
\end{equation*}

The above generating function allows us to present further linear combination of this multiple Laguerre type polynomials in terms of Multiple Laguerre polynomials analogue of  (\ref{LinerCombLaguer2}). 
The Laguerre type $d$-orthogonal polynomials given in the example \ref{Ex1} have been studied and evoked in many places see for instance 
\cite{CheikhDouakLaguerre,CheikhRomdSym,CheikhRomdSymBrenke}.  Thereof, the generating function shows that the family of polynomials generated by (\ref{NLag}) is quite different from  that  given in the example \ref{Ex1} as well as from the example studied in \cite{Varma}.

In most cases, since the above polynomials and their derivatives are both d-orthogonal, then one can explicitly determine the respective measures of the $d$-orthogonality using either Pearson equation \cite{DouakCar} or the quasi-monomiality principle \cite{CheikhZag2}. Notice that the latter idea does not require any information on the derivative sequence, for this end we shall use the latter idea in the next family which converges toward the above polynomials as $w$ goes to zero and the computations are almost the same (see \cite[Lemma 2.7]{CheikhZag2}). But this does not preclude mentioning some properties, compared by classical Laguerre polynomials, we shall take $\alpha+1=-\beta$ and denote 
$\pi_{d-1}(t;b_i)=b_{0}+b_{1}t+b_{2}\frac{t^2}{2!}+...+b_{d-1}\frac{t^{d-1}}{(d-1)!}$. 
In this case, since $t\pi'_{d-1}(t;b_i)=\pi_{d-1}(t;ib_i)$ we can show that
\begin{equation*}
x\frac{\partial}{\partial x}G\left(x,t\right) 
+(1-at)\pi_{d-1}(t;ib_i)G\left(x,t\right)
=t\frac{\partial}{\partial t}G\left(x,t\right) -at^{1-\alpha}\frac{\partial}{\partial t}
\left\{t^{\alpha+1}G\left(x,t\right)\right\}
\end{equation*}
from which we deduce the structure relation (\ref{SR2}) in closed form
\begin{equation*}
\begin{array}{cl}
xP'_{n}(x)=nP_{n}(x)-n\left\{b_1+a(n+\alpha)\right\}P_{n-1}(x)
\\+\displaystyle\sum\limits_{i=2}^{d}
\left\{ a\frac{b_{i-1}}{(i-2)!}-\frac{b_{i}}{(i-1)!}\right\}
\left\langle  n\right\rangle _{i}P_{n-i}(x)
\end{array}
\end{equation*}
which reduces in turn to classical Laguerre with $d=1$, i.e., when $\pi_{d-1}(t;b_i)=0$.
We would like to mention further that the differential equation satisfied by these types of polynomials is completely ignored in literature except for Appell case. 
In our point of view, differential equations can be constructed using the linear combination (\ref{NCC}) as well as some structure relations (further results of this idea will be presented in forthcoming papers). For the above Laguerre case, it can be obtained simply by taking $w=0$ in (\ref{DE2}) and replace $\Delta_{w}$ by $d/dx$ (see next subsection for more details). 


\subsection{A discrete solution: Mittag-Leffler type polynomials}\label{sec:ML}

Now let us suppose that $\mathcal{L}=\Delta_w$.	
The $\Delta_{w}$ difference operator is defined as follows
\begin{equation*}
\Delta_{w}f(x)=\frac{f(x+w)-f(x)}{w}.
\end{equation*}

Next, we shall prove that the discrete solution of (\ref{NCCd}) is a generalization of Mittag-Leffler polynomials \cite{BatemanML} which seems to be new.

Let denote by $K(x,t)=\sum\nolimits_{n\geq 0}P_n(x)\frac{t^n}{n!}$ the respective exponential generating function of $\left\{P_n\right\}_{n\geq 0}$. According to (\ref{NCC}), in this paper we shall assume that $\lambda_{n,1}=n\alpha$ and  $w=\alpha -\beta$. 
It is straightforward to transform (\ref{NCCd}) to the following initial value problem
\begin{equation*}
K(x,0)=1, \ \ \ \ (1-\alpha t)\Delta_w K(x,t)=t K(x,t).
\end{equation*}
It  is not difficult to show that the unique solution of the above equation is 
\begin{equation}
K(x,t)=\left(\frac{1-\beta t}{1-\alpha t}\right)^{x/w} A(t), \ \ \ A(0)=1.
\label{GFPollaczk}
\end{equation}

Now assume that $A'(t)/A(t)=\sum\nolimits_{k\geq 0} b_k \frac{t^k}{k!}$. 
Therefore, the partial derivative of (\ref{GFPollaczk}) with respect to $t$ gives
\begin{equation*}
(1-\alpha t)(1-\beta t)K'_t(x,t)=
\left[x+(1-\alpha t)(1-\beta t)A'(t)/A(t)\right]K(x,t),
\end{equation*}
from which it follows
\begin{align}\displaystyle
P_{n+1}(x)&\displaystyle=\left[x+(\alpha+\beta)n+b_0\right]P_n(x)
-n\left[(n-1)\alpha\beta +(\alpha+\beta)b_0-b_1\right]P_{n-1}(x)\nonumber
\\
&+\sum\nolimits_{k=2}^n\binom{n}{k}\left[b_k-(\alpha+\beta)kb_{k-1}
+\alpha\beta k(k-1)b_{k-2}\right]P_{n-k}(x).\label{MRRL1}
\end{align}
On the other hand, since $\left\{ P_{n}\right\} _{n\geq 0} $ is an $d$-OPS, then we must have $b_k=0$ for $k\geq d-1$. 

It is worthy to notice that the generating function (\ref{GFPollaczk}) reduces to $\Delta_w$-Appell $d$-OPS (Charlier $d$-OPS \cite{CheikhZag1}) if $\alpha$ or $\beta$ is zero, and to Mittag-Leffler's generating function \cite{BatemanML} in the case $d=1$ with $\alpha=-\beta=1$.

Moreover,  since sequences generated by the above generating function and their derivatives are both $d$-OPS, then it is more convenient to write down the corresponding recurrence of the derivative sequence. 
Then by acting the operator $\Delta _{w}$ on (\ref{MRRL1}) 
($w=\alpha -\beta $) and making use of (\ref{NCCd}), we obtain upon writing $b_{k}\equiv 0$ for $k\geq d-1$ the recurrence of the derivative sequence
\begin{align}\displaystyle
Q_{n+1}(x)&=\left[x+(\alpha+\beta)n+b_0+\alpha\right]Q_n(x)
-n\left[n\alpha\beta +(\alpha+\beta)b_0-b_1\right]Q_{n-1}(x)\nonumber
\\
&+\sum\nolimits_{k=2}^n\binom{n}{k}\left[b_k-(\alpha+\beta)kb_{k-1}
+\alpha\beta k(k-1)b_{k-2}\right]Q_{n-k}(x).\label{MRRL2}
\end{align}


Let us now mention some properties of the obtained polynomials. 
\begin{proposition}
	The above family of polynomials satisfies the following recurrences
	\begin{align}
	P_n(x+w)&=Q_n(x)-n\beta Q_{n-1}(x), \label{SR5} \\
	P_n(x)-\beta nP_{n-1}(x)&=P_n(x+w)-\alpha nP_{n-1}(x+w), \label{SR7} \\
	wQ_n(x)&=\alpha P_n(x+w)-\beta P_n(x), \label{SR3} \\
	\Delta_w\big\{P_{n+1}(x)P_n(x)\big\}&
	=(n+1)P_{n}(x+w)Q_n(x)+nP_{n+1}(x)Q_{n-1}(x)   \label{SR4} \\
	=(n+1)Q_n^2(x)&+nP_{n+1}(x)Q_{n-1}(x)-n(n+1)\beta Q_n(x)Q_{n-1}(x),\nonumber\\
	(x-c)Q_n=P_{n+1}(x)-&(c+b_{0}+\beta n)P_n(x)
	-\sum_{i=1}^{d-1}\binom{n}{i}\left(\beta ib_{i-1}-b_{i}\right)P_{n-i}(x). \label{SR6}
	\end{align}
\end{proposition}
\begin{proof}
	The recurrence coefficients of (\ref{MRRL1}) and (\ref{MRRL2}) show that the second structure relation in (\ref{NCC}) reduces to (\ref{SR5}) while (\ref{SR2}) takes the form (\ref{SR6}).
	
	Let us remark further that from 
	\begin{align*}
	(1-\beta t)K(x,t)=(1-\alpha t)K(x+w,t)
	\end{align*}
	we deduce (\ref{SR7}). Accordingly, we have  
	\begin{align*}
	w\Delta_{w}P_n(x)=nwQ_{n-1}(x)=P_n(x+w)-P_n(x)=\alpha nP_{n-1}(x+w)-\beta nP_{n-1}(x),
	\end{align*}
	from which (\ref{SR3}) follows.
	
	Let us now prove (\ref{SR4}). Remark that from the following fact
	\begin{align*}
	w\Delta_w\big\{P_{n+1}(x)P_n(x)\big\} =P_{n+1}(x+w)P_n(x+w)-P_{n+1}(x)P_n(x),
	\end{align*}
	we can eliminate the factor $P_{n+1}(x+w)P_n(x+w)$ by multiplying both sides of  (\ref{SR7}) by $P_{n-1}(x+w)$ and $n\rightarrow n+1$ then, replace the obtained result in the latter equality above and use also (\ref{SR3}) to deduce the desired result which in turn could be simplified to the second equality using (\ref{SR5}). 
\end{proof}


Let us now turn to the difference equation satisfied by the above polynomials. We would like to mention that property (\ref{NCCd}) makes the construction of the respective differential/difference equation very simple. Indeed, it suffices to apply $d+1$ times the operator $\Delta_{w}$ to the recurrence relation satisfied by the polynomials  and use in each time the connection (\ref{NCCd}) to move from $n$ to $n-1$. 
For convenience let us denote the recurrence coefficients of (\ref{MRRL1})  satisfied by the polynomials $\left\{P_n\right\}_{n\geq 0}$ by $\beta_{n}$ and $\gamma_{n}^i$. 
Then the latter sequence satisfies the following difference equation which can be easily proved by induction on $k$ using (\ref{NCCd}) and the binomial property $\binom{n}{i}+\binom{n}{i+1}=\binom{n+1}{i+1}$ with some computations.

\begin{theorem}
	The $d$-OPS solution of (\ref{NCCd}) with $\lambda_{n,1}=n\alpha$ satisfies 
	\begin{align}
	P_{n-k+1}(x)&=\left[x+kw-k\alpha(n-k+2)-\beta_{n}\right]P_{n-k}(x)\nonumber\\
	&+\sum_{i=1}^{k}\left\{\alpha^i\left[\binom{k}{i}(x+kw+\alpha-\beta_{n})
	-\alpha\binom{k+1}{i+1}(n-k+i+2)\right]\right.\nonumber\\
	&\left.-\sum_{j=0}^{i-1}\binom{k-1-j}{i-1-j}\alpha^{i-1-j}\frac{\gamma_{n-j}^{d-1-j}}{\left\langle n\right\rangle _{j+1}}\right\}\Delta_{w}^iP_{n-k}(x)\label{DE1}\\
	&-\sum_{i=k}^{d-1}\frac{\gamma_{n-i}^{d-1-i}}{\left\langle n\right\rangle_{k}}\Delta_{w}^kP_{n-i-1}(x).\nonumber
	\end{align}
\end{theorem}

From the latter result we merely deduce the following difference equation

\begin{corollary}
	The d-OPS solution of (\ref{NCCd}) satisfies the following (d+1)-order difference equation
	\begin{align}
	(n-d)P_{n-d}(x)&=\left[x+(d+1)w-(d+1)\alpha(n-d+1)-\beta_{n}\right]\Delta_{w}P_{n-d}(x)\nonumber\\
	&\hspace{-12mm}+\sum_{i=1}^{d}\left\{\alpha^i\left[\binom{d}{i}(x+(d+1)w-\beta_{n})
	-\alpha\binom{d+1}{i+1}(n-d+i+1)\right]\right.\nonumber\\
	&\left.-\sum_{j=0}^{i-1}\binom{d-1-j}{i-1-j}\alpha^{i-1-j}\frac{\gamma_{n-j}^{d-1-j}}{\left\langle n\right\rangle_{j+1}}\right\}\Delta_{w}^{i+1}P_{n-d}(x).\label{DE2}
	\end{align}
\end{corollary}

\begin{proof}
	Take first $k=d-1$ in (\ref{DE1}), then apply two times $\Delta_{w}$ on both sides of (\ref{DE1}) together with (\ref{NCCd}) in each time to deduce explicitly, after some straightforward calculations, the difference equation. 
\end{proof}


In this case the explicit form of the polynomials 
generated by (\ref{SZ1}) may be written in terms of generalized falling factorial. 
Let us denote  $b_{i-1}/i!=a_i$ and $\exp\left\{a_0\right\}=1$, then

\begin{theorem}
	We have%
	\begin{equation}
	\begin{array}{l}
	\displaystyle P_{n}\left( x\right) \displaystyle=\sum\limits_{s=0}^{n}
	\sum\limits_{\begin{array}{c}
		k_1+...+(d-1)k_{d-1}=m\\0\leq m\leq s\end{array}} 
	\binom{n} {k_1,...,k_{d-1},n-s,s-m,m}
	\vspace*{2mm}\\\displaystyle\times
	\left(a_1\right)^{k_1}\left(a_{2}\right)^{k_2} ...\left(a_{d-1}\right) ^{k_{d-1}}
	\left(\frac{\beta}{\alpha}\right)^{s}\left(-\alpha\right)^{n}\left(-\beta\right)^{m}
	x \left\langle  x+(n-s-1)w|w\right\rangle _{n-m-1}. 
	\end{array}\label{SZ4}
	\end{equation}
\end{theorem}

The explicit form of the polynomial is a direct consequence of the following
result together with the Cauchy product of power series

\begin{lemma}
	With $w=\alpha -\beta $, we have%
	\begin{equation}
	\hspace{-5mm}\left( \frac{1-\beta t}{1-\alpha t}\right)^{x/w}=\sum\limits_{n=0}^{\infty }\sum\limits_{k=0}^{n}
	\binom{n}{k}\left(\frac{\beta}{\alpha}\right)^{k}
	x \left\langle  x+(n-k-1)w|w\right\rangle  _{n-1}\frac{(-\alpha t)^n}{n!}\label{SZ5}
	\end{equation}
\end{lemma}

\begin{proof}
	Taking into account the following power series
	\begin{equation*}
	\left( 1-\beta t\right) ^{x/w}=\sum\limits_{n=0}^{\infty }\left( -\beta \right) ^{n}\left\langle  x|w\right\rangle  _{n}\frac{t^{n}}{n!}
	\end{equation*}%
	\begin{equation*}
	\left( 1-\alpha t\right) ^{-x/w}=\sum\limits_{n=0}^{\infty }\left( -\alpha\right) ^{n}
	\left\langle  x(n-1)w|w\right\rangle_{n}\frac{t^{n}}{n!}
	\end{equation*}
	the convolution follows from the Cauchy product of two series together with
	\begin{equation*}
	\left\langle  x|w\right\rangle _m\left\langle  x+(n-m-1)w|w\right\rangle  _{n-m}=x \left\langle  x+(n-m-1)w|w\right\rangle  _{n-1},
	\end{equation*}
	or $\left\langle  x|w\right\rangle _{n-m}\left\langle  x+(m-1)w|w\right\rangle  _{m}=x \left\langle  x+(m-1)w|w\right\rangle  _{n-1}$. 
\end{proof}

\section{The dual sequence}\label{sec:dual}

In the applications, it might be useful sometimes to have an explicit expression for the moments to interpret,  combinatorially or physically, the corresponding family of polynomials in one hand. For this end, we shall give here some information about the moments at first. Therefore, starting from the generating function we can identify the expression of polynomials as well as their inversion formulas by comparing the coefficients of $t$. 

On the other hand, since these polynomial sequences and their derivatives are both $d$-OPS, then it might be possible  to use Pearson equation \cite{DouakCar} to determine the dual sequence (i.e. $d$-dimensional vector of linear forms) with respect to which the polynomials are $d$-orthogonal. Unfortunately, to determine the dual sequence's elements $\left\{\varphi_r, \ 0\leq r\leq d-1\right\}$,  the latter fact leads, in general, to look at solutions of linear differential equations of order exactly $d$ 
\begin{align*}
\sum_{i=0}^{d}\pi_{i+d+r,i}(x)\frac{d^i}{dx}\varphi_r(x)=0,  
\hspace*{.5cm} \text{for each } \hspace*{.5cm} 0\leq r\leq d-1,
\end{align*}
 with polynomial coefficients but of degrees greater than $i$ (at most $d+r+i$). The above differential equation comes from Pearson equation by direct computations.  

Moreover, if the generating function is of Brenke type (the lucky and the faster case), then we can use the Laplace, $h$-Laplace, $q$-Laplace transformations (discrete time scales) and their inverse to compute the measures of orthogonality. 

Besides, a practical technique is the quasi-monomiality principle which has been developed by Ben Cheikh and his collaborators \cite{CheikhQM2,CheikhQM1} to determine the dual sequence of polynomials mainly in the discrete case \cite{CheikhZag2}. 
 
To start with, let us lake $-c_i=b_{i-1}/(i-1)$ and $c_0=0$, then from the generating function (\ref{SZ1}) we have
 \begin{equation}
\left(\frac{1-\beta t}{1-\alpha t}\right)^{x/w}= \exp \left\{ \sum_{i=1}^{d-1}c_{i}t^{i}\right\}
\sum\nolimits_{n\geq 0}P_n(x)\frac{t^n}{n!}\label{SZ3}
\end{equation}
now expand the right hand side of (\ref{SZ3}) in powers of $t$ and then identify the coefficient of $t^n$ in both sides we deduce using (\ref{SZ5}) the following 
 \begin{equation*}
 \begin{array}{l}
\sum_{k=0}^{n}\binom{n_1+2n_2+...+(d-1)n_{d-1}}{n_1,\ n_2,...,\ n_{d-1},\ k}
c_1^{n_1}...c_{d-1}^{n_{d-1}} P_k(x) \vspace{2mm}\\
 =\sum\limits_{k=0}^{n}
 \binom{n}{k}\left(\frac{\beta}{\alpha}\right)^{k}
 x \left\langle  x+(n-k-1)w|w\right\rangle_{n-1}(-\alpha )^n.
  \end{array}
 \end{equation*}
If we denote by $\left\{\varphi_r\right\}_{r\geq 0}$ the dual sequence of Mittag-Leffer $d$-orthogonal polynomials $\left\{P_n\right\}_{n\geq 0}$, we infer that
\begin{proposition}
The moments satisfy $\left\langle \varphi_r,x^{n}\right\rangle=0$ for $n<r$ and the following finite linear recursion for $n\geq r$ 
 \begin{equation*}
\binom{n_1+2n_2+...+(d-1)n_{d-1}}{n_1,\ n_2,...,\ n_{d-1},\ r}
c_1^{n_1}...c_{d-1}^{n_{d-1}} =\sum\limits_{k=r}^{n}
\binom{n}{k}\left(\frac{\beta}{\alpha}\right)^{k}
(-\alpha )^n \left\langle \varphi_r,x^{k}\right\rangle.
\end{equation*}
\end{proposition}
 
For the dual sequence, it has been proved in the discrete case that the latter could be obtained via
\begin{equation}
\langle \varphi_r,f\rangle =\frac{\sigma ^{r}}{r!A(\sigma )}f(0)=\frac{1}{r!}%
\sigma ^{r}\exp \left\{ \sum_{i=1}^{d-1}a_{i}\sigma ^{i}\right\} f(0)\ \ 
\text{for}\ \ 0\leq r\leq d-1,  \label{G15}
\end{equation}
where $\sigma:=\sigma_x$ is the lowering operator, i.e., $\sigma G(x,t)=tG(x,t)$ with $G(x,t)=G_0(x,t)A(t)$ and where we have denoted by $a_k=-b_{k-1}$. 
Therefore, according to \cite{CheikhQM1,CheikhZag2}, the operator $\sigma $ is given by 
\begin{equation*}
\sigma :=\frac{e^{wD}-1}{\alpha e^{wD}-\beta }=\frac{\Delta _{w}%
}{1+\alpha \Delta _{w}},
\end{equation*}
and by the binomial theorem we have%
\begin{equation*}
\begin{array}{cl}
\langle\varphi_r,f\rangle & =\frac{1}{r!}\sum\limits_{n_{1},...,n_{d-1}\geq
	0}^{\infty }\frac{a_{1}^{n_{1}}}{n_{1}!}...\frac{a_{d-1}^{n_{d-1}}}{n_{d-1}!}%
\frac{\Delta _{w}^{l+r}}{\left( 1+\alpha \Delta _{w}\right) ^{l+r}}f(0)%
\vspace{0.2cm} \\ 
&  =\frac{1}{r!}\sum\limits_{n_{1},...,n_{d-1}\geq 0}^{\infty }\frac{%
	a_{1}^{n_{1}}}{n_{1}!}...\frac{a_{d-1}^{n_{d-1}}}{n_{d-1}!}%
\sum\limits_{k=0}^{\infty }\frac{\left\langle  l+r\right\rangle _{k}\left( -\alpha
	\right) ^{k}\Delta _{w}^{l+r+k}}{k!}f(0),%
\end{array}%
\end{equation*}%
with $l:=n_{1}+2n_{2}+...+\left( d-1\right) n_{d-1}$ which can be written using the expansion 
\begin{equation*}
\Delta _{w}^{n}f\left( 0\right) =\left( \frac{-1}{w}\right)
^{n}\sum\limits_{j=0}^{n}\binom{n}{j}\left( -1\right) ^{j}f\left( wj\right) ,
\end{equation*}%
in the following form
\begin{equation*}
\begin{array}{l}
r!\langle \varphi_r,f\rangle=\\\sum\limits_{k,n_{1},...,n_{d-1}\geq
	0}^{\infty }\frac{a_{1}^{n_{1}}}{n_{1}!}...\frac{a_{d-1}^{n_{d-1}}}{n_{d-1}!}%
\left( \frac{-1}{w}\right) ^{l+r+k}\frac{\left\langle  l+r\right\rangle_{k}\left(
	-\alpha \right) ^{k}}{k!}\sum\limits_{s=0}^{l+r+k}\binom{l+r+k}{s}\left(
-1\right) ^{s}f(ws),
\end{array}
\end{equation*}
by writing the sum from 0 to $l+r+k$ as two sums, the first ends at $l+r-1$ and the second starts from $l+r$, using also series manipulation \cite[p. 100-102]{Sriva}, we get the following  
\begin{proposition} The elements of the dual sequence of Mittag-Leffler $d$-OPS satisfy 
\begin{equation*}
\begin{array}{cl}
\left\langle \varphi_r,f\right\rangle &  =\frac{1}{r!}\sum%
\limits_{k,n_{1},...,n_{d-1}\geq 0}^{\infty }\frac{a_{1}^{n_{1}}}{n_{1}!}...%
\frac{a_{d-1}^{n_{d-1}}}{n_{d-1}!}\left( \frac{-1}{w}\right) ^{l+r+k}\frac{%
	\left\langle  l+r\right\rangle_{k}\left( -\alpha \right) ^{k}}{k!}\vspace{0.15cm} \\ 
&\hspace{-6mm} \times \left[ \sum\limits_{s=0}^{k}\binom{l+r+k}{s+l+r}\left( -1\right)
^{s+l+r}f(w\left( s+l+r\right) )+\sum\limits_{s=0}^{l+r-1}\binom{l+r+k}{s}%
\left( -1\right) ^{s}f(ws)\right].
\end{array}%
\end{equation*}
\end{proposition}

In order to specify the latter sums, let us denote the first and second sum by $A$ and $B$, respectively.
Taking account the form of the integer $l$, we shall write $B$ as a $d-1$
partial sums each one from $(i-1)n_{i-1}$ to $in_i-1$ as bellow. Therefore, using the following 
\begin{equation*}
\sum_{n=0}^{\infty }\sum_{k=0}^{dn+r-1}A\left( k,n\right) 
=\sum_{n,k=0}^{\infty }A\left( k,n+\left[ \left( k+1-r\right) /d\right]	\right)
\end{equation*}
with $l_{0}=0$ and $l_{i}:=n_{1}+2n_{2}+...+in_{i}$, for $1\leq i\leq d-2$, we have
\begin{equation*}
\begin{array}{l}
r!B_{i}  :=\sum\limits_{k,n_{1},...,n_{d-1}\geq 0}^{\infty }\frac{%
	a_{1}^{n_{1}}}{n_{1}!}...\frac{a_{d-1}^{n_{d-1}}}{n_{d-1}!}\left( \frac{-1}{w%
}\right) ^{l+r+k}\frac{\left\langle  l+r\right\rangle_{k}\left( -\alpha \right) ^{k}}{k!} \vspace{0.15cm} \\ 
\ \ \ \ \times \sum\limits_{s=0}^{in_{i}-1}\binom{l+r+k}{s+l_{i-1}}\left( -1\right)
^{s+l_{i-1}}f(w\left( s+l_{i-1}\right) \vspace{0.15cm} \\ 
 =\sum\limits_{s,k,n_{1},...,n_{d-1}\geq 0}^{\infty }\frac{a_{1}^{n_{1}}}{%
	n_{1}!}...\frac{a_{i-1}^{n_{i-1}}}{n_{i-1}!}\frac{a_{d-2}^{n_{i}+\left[
		\left( s+1\right) /i\right] }}{\left( n_{i}+\left[ \left( s+1\right) /i%
	\right] \right) !}\frac{a_{i+1}^{n_{i+1}}}{n_{i+1}!}...\frac{%
	a_{d-1}^{n_{d-1}}}{\left( n_{d-1}\right) !}\left( \frac{-1}{w}\right)
^{l+r+k+i\left[ \left( s+1\right) /i\right] }\vspace{0.15cm} \\ 
\times \frac{\left\langle  l+r+i\left[ \left( s+1\right) /i\right] \right\rangle_{k}\left( -\alpha \right) ^{k}}{k!}\binom{l+r+k+i\left[ \left( s+1\right) /i%
	\right] }{s+l_{i-1}}\left( -1\right) ^{s+l_{i-1}}f(w\left( s+l_{i-1}\right) ,%
\end{array}%
\end{equation*}%
and 
\begin{equation*}
\begin{array}{cl}
r!B_{d-1} & :=\sum\limits_{k,n_{1},...,n_{d-1}\geq 0}^{\infty }\frac{
	a_{1}^{n_{1}}}{n_{1}!}...\frac{a_{d-1}^{n_{d-1}}}{n_{d-1}!}\left( \frac{-1}{w%
}\right) ^{l+r+k}\frac{\left\langle  l+r\right\rangle_{k}\left( -\alpha \right) ^{k}}{k!}%
\vspace{2mm} \\ 
& \times \sum\limits_{s=0}^{\left( d-1\right) n_{d-1}+r-1}\binom{l+r+k}{%
	s+l_{d-2}}\left( -1\right) ^{s+l_{d-2}}f(w\left( s+l_{d-2}\right) \vspace{%
	2mm} \\ 
& =\sum\limits_{s,k,n_{1},...,n_{d-1}\geq 0}^{\infty }\frac{a_{1}^{n_{1}}}{%
	n_{1}!}...\frac{a_{d-2}^{n_{d-2}}}{n_{d-2}!}\frac{a_{d-1}^{n_{d-1}+\left[
		\left( s+1-r\right) /\left( d-1\right) \right] }}{\left( n_{d-1}+\left[
	\left( s+1-r\right) /\left( d-1\right) \right] \right) !} \vspace{2mm}\\ & \times\left( \frac{-1}{w}%
\right) ^{l+r+k+\left( d-1\right) \left[ \left( s+1-r\right) /\left(
	d-1\right) \right] }
\frac{\left\langle  l+r+\left( d-1\right) \left[ \left( s+1-r\right)
	/\left( d-1\right) \right] \right\rangle  _{k}\left( -\alpha \right) ^{k}}{k!} \vspace{2mm}
\\ &\times \binom{l+r+k+\left( d-1\right) \left[ \left( s+1-r\right) /\left( d-1\right) %
	\right] }{s+l_{d-2}}\left( -1\right) ^{s+l_{d-2}}f(w\left( s+l_{d-2}\right) .%
\end{array}%
\end{equation*}

Let us now return to the first sum denoted by $A$. We have, using $\left(
k+s\right) !=\left\langle  s+1\right\rangle_{k}s!$, that%
\begin{equation*}
\begin{array}{l}
r!A:= \\\hspace{-3mm}\sum\limits_{s,k,n_{1},...,n_{d-1}\geq 0}^{\infty }\frac{%
	a_{1}^{n_{1}}}{n_{1}!}...\frac{a_{d-1}^{n_{d-1}}}{n_{d-1}!}\left( \frac{1}{w%
}\right) ^{l+r+k+s}\frac{(-1) ^{s} \alpha^{k+s}
	\left\langle  l+r\right\rangle_{k+s}\left( l+r+s+k\right) !}{\left( s+l+r\right)
	!\left( s+k\right) !k!}f(w(s+l+r)) \vspace{0.15cm} \\ 
  =\sum\limits_{s,k,n_{1},...,n_{d-1}\geq 0}^{\infty }\frac{a_{1}^{n_{1}}}{%
	n_{1}!}...\frac{a_{d-1}^{n_{d-1}}}{n_{d-1}!}\left( \frac{\alpha }{w}\right)
^{l+r+s}\left( -\alpha \right) ^{s}f(w\left( s+l+r\right) \vspace{0.15cm} \\ 
 \times \ _{2}F_{1}\left( \left. 
\begin{tabular}{cc}
$l+r+s,$ & $l+r+s+1$ \\ 
\multicolumn{2}{c}{$s+1$}%
\end{tabular}%
\right\vert \frac{\alpha }{w}\right) ,%
\end{array}%
\end{equation*}%
and 
\begin{equation*}
\begin{array}{cl}
r!B_{1} & =\sum\limits_{s,k,n_{1},...,n_{d-1}\geq 0}^{\infty }\frac{%
	a_{1}^{n_{1}+s+1}}{\left( n_{1}+s+1\right) !}...\frac{a_{d-1}^{n_{d-1}}}{%
	n_{d-1}!}\left( -1\right) ^{s}\left( \frac{-1}{w}\right) ^{l+r+s+1}\frac{%
	\left\langle  l+r+2\right\rangle_{s}}{s!}f(ws)\vspace{2mm} \\ 
& \times \ _{2}F_{1}\left( \left. 
\begin{tabular}{cc}
$l+r+s+1,$ & $l+r+s+2\vspace{0.2cm}$ \\ 
\multicolumn{2}{c}{$l+r+2$}%
\end{tabular}%
\right\vert \frac{\alpha }{w}\right) .%
\end{array}%
\end{equation*}

Now we use, to get the value of $B_{i}$, for $2\leq i\leq d-1$, the fact that%
\begin{equation*}
\left( dk+s\right) !=s!d^{dk}
\left\langle  \frac{s+1}{d}\right\rangle_{k}
\left\langle  \frac{s+2%
}{d}\right\rangle_{k}...\left\langle  \frac{s+d}{d}\right\rangle_{k}
\end{equation*}%
and properties of factorial to obtain for $2\leq i\leq d-2$ 
\begin{equation*}
\begin{array}{cl}
\hspace{-2mm}r!B_{i} & =\sum\limits_{s,k,n_{1},...,n_{d-1}\geq 0}^{\infty }\frac{%
	a_{1}^{n_{1}}}{n_{1}!}...\frac{a_{i-1}^{n_{i-1}}}{n_{i-1}!}\frac{%
	a_{d-2}^{n_{i}+\left[ \left( s+1\right) /i\right] }}{\left( n_{i}+\left[
	\left( s+1\right) /i\right] \right) !}\frac{a_{i+1}^{n_{i+1}}}{n_{i+1}!}...%
\frac{a_{d-1}^{n_{d-1}}}{\left( n_{d-1}\right) !}\left( \frac{-1}{w}\right) ^{l+r+i	\left[ \left( s+1\right) /i\right] }\vspace{4mm} \\ 
& \times \left( -1\right) ^{s+l_{i-1}}f(w\left( s+l_{i-1}\right)) \frac{\left(
	l+r+i\left[ \left( s+1\right) /i\right] \right) !}{\left( s+l_{i-1}\right)
	!\left( l-l_{i-1}+r+i\left[ \left( s+1\right) /i\right] -s\right) !}\vspace{4mm} \\ 
& \times \ _{2}F_{1}\left( \left. 
\begin{tabular}{cc}
$l+r+i\left[ \left( s+1\right) /i\right] ,$ & $l+r+1+i\left[ \left(
s+1\right) /i\right] \vspace{2mm}$ \\ 
\multicolumn{2}{c}{$l-l_{i-1}+r+1+i\left[ \left( s+1\right) /i\right] -s$}%
\end{tabular}%
\right\vert \frac{\alpha }{w}\right) ,%
\end{array}%
\end{equation*}%
and 
\begin{equation*}
\begin{array}{cl}
\hspace*{-4mm}r!B_{d-1} & =\sum\limits_{s,k,n_{1},...,n_{d-1}\geq 0}^{\infty }\frac{%
	a_{1}^{n_{1}}}{n_{1}!}...\frac{a_{d-1}^{n_{d-1}+\left[ \left( s+1-r\right)
		/\left( d-1\right) \right] }\left( -1\right) ^{s+l_{i-1}}}{\left( n_{d-1}+%
	\left[ \left( s+1-r\right) /\left( d-1\right) \right] \right) !}\left( \frac{-1}{w}\right) ^{l+r+\left( d-1\right) \left[ \frac{s+1-r}{d-1} \right] } \vspace{4mm}\\ 
&\times  \frac{\left( l+r+\left( d-1\right) \left[ \left( s+1-r\right) /\left( d-1\right) %
	\right] \right) !}{\left( s+l_{i-1}\right) !\left( l-l_{i-1}+r+\left(
	d-1\right) \left[ \left( s+1-r\right) /\left( d-1\right) \right] -s\right) !}f(w\left( s+l_{i-1}\right))
\vspace{4mm} \\ 
& \times \ _{2}F_{1}\left( \left. 
\begin{tabular}{cc}
$l+r+\left( d-1\right) \left[ \frac{s+1-r}{d-1}\right] ,$ & $l+r+1+\left(
d-1\right) \left[ \frac{s+1-r}{d-1}\right] \vspace{0.2cm}$ \\ 
\multicolumn{2}{c}{$l-l_{i-1}+r+1+\left( d-1\right) \left[ \frac{s+1-r}{d-1}%
	\right] -s$}%
\end{tabular}%
\right\vert \frac{\alpha }{w}\right) .%
\end{array}%
\end{equation*}

We can simplify the expression of $B_{i}$ by using the properties of the
integer part. Indeed, for $2\leq i\leq d-2$, we can write%
\begin{equation*}
\begin{array}{cl}
r!B_{i} & =\sum\limits_{p=0}^{i-2}\ \sum\limits_{s,k,n_{1},...,n_{d-1}\geq
	0}^{\infty }\frac{a_{1}^{n_{1}}}{n_{1}!}...\frac{a_{i-1}^{n_{i-1}}}{n_{i-1}!}%
\frac{a_{d-2}^{n_{i}+s}}{\left( n_{i}+s\right) !}\frac{a_{i+1}^{n_{i+1}}}{%
	n_{i+1}!}...\frac{a_{d-1}^{n_{d-1}}}{\left( n_{d-1}\right) !}\vspace{3mm}
\\ 
&  \times \left( -1\right) ^{is+p+l_{i-1}}\left( \frac{-1}{w}\right) ^{l+r+is}%
\frac{\left( l+r+is\right) !}{\left( is+p+l_{i-1}\right) !\left(
	l-l_{i-1}+r-p\right) !}f(w\left( is+p+l_{i-1}\right) \vspace{0.2cm} \\ 
&  \times \ _{2}F_{1}\left( \left. 
\begin{tabular}{cc}
$l+r+is,$ & $l+r+1+is\vspace{0.2cm}$ \\ 
\multicolumn{2}{c}{$l-l_{i-1}+r+1-p$}%
\end{tabular}%
\right\vert \frac{\alpha }{w}\right) \vspace{0.2cm} \\ 
&  +\sum\limits_{s,k,n_{1},...,n_{d-1}\geq 0}^{\infty }\frac{a_{1}^{n_{1}}}{%
	n_{1}!}...\frac{a_{i-1}^{n_{i-1}}}{n_{i-1}!}\frac{a_{d-2}^{n_{i}+s+1}}{%
	\left( n_{i}+s+1\right) !}\frac{a_{i+1}^{n_{i+1}}}{n_{i+1}!}...\frac{%
	a_{d-1}^{n_{d-1}}}{\left( n_{d-1}\right) !}\left( \frac{-1}{w}\right)
^{l+r+is+i}\vspace{0.2cm} \\ 
&  \times \left( -1\right) ^{is+i-1+l_{i-1}}f(w\left( is+i-1+l_{i-1}\right) 
\frac{\left( l+r+is+i\right) !}{\left( is+i-1+l_{i-1}\right) !\left(
	l-l_{i-1}+r+1\right) !}\vspace{0.2cm} \\ 
& \times \ _{2}F_{1}\left( \left. 
\begin{tabular}{cc}
$l+r+is+i,$ & $l+r+1+is+i\vspace{0.2cm}$ \\ 
\multicolumn{2}{c}{$l-l_{i-1}+r+2$}%
\end{tabular}%
\right\vert \frac{\alpha }{w}\right) .%
\end{array}%
\end{equation*}

For $B_{d-1}$, since $2-d\leq p+1-r\leq d-1$, we have 
\begin{equation*}
\left[ \frac{\left( d-1\right) s+p+1-r}{d-1}\right] =\left\{ 
\begin{array}{ll}
s-1,\  & \text{for \ }2-d\leq p+1-r<0\text{,} \\ 
s, & \text{for \ }0\leq p+1-r\leq d-2\text{,} \\ 
s+1, & \text{for \ }p+1-r=d-1\text{,}%
\end{array}%
\right. 
\end{equation*}%
then, for \ $2-d\leq p+1-r<0$, we have%
\begin{equation*}
\begin{array}{cl}
r!B_{d-1} & =\sum\limits_{s,k,n_{1},...,n_{d-1}\geq 0}^{\infty }\frac{%
	a_{1}^{n_{1}}}{n_{1}!}...\frac{a_{d-1}^{n_{d-1}+s-1}\left( -1\right)
	^{\left( d-1\right) s+p+l_{i-1}}}{\left( n_{d-1}+s-1\right) !}\left( \frac{-1%
}{w}\right) ^{l+r+\left( d-1\right) \left( s-1\right) }\vspace{0.15cm} \\ 
& \times \frac{\left( l+r+\left( d-1\right) \left( s-1\right) \right) !}{%
	\left( \left( d-1\right) s+p+l_{i-1}\right) !\left( l-l_{i-1}+r+1-d-p\right)
	!}f(w\left( \left( d-1\right) s+p+l_{i-1}\right) \vspace{0.2cm} \\ 
& \times \ _{2}F_{1}\left( \left. 
\begin{tabular}{cc}
$l+r+\left( d-1\right) \left( s-1\right) ,$ & $l+r+1+\left( d-1\right)
\left( s-1\right) \vspace{0.2cm}$ \\ 
\multicolumn{2}{c}{$l-l_{i-1}+r+1-d-p$}%
\end{tabular}%
\right\vert \frac{\alpha }{w}\right) 
\end{array}%
\end{equation*}%
and for $0\leq p+1-r\leq d-2$, we obtain%
\begin{equation*}
\begin{array}{cl}
r!B_{d-1} & =\sum\limits_{s,k,n_{1},...,n_{d-1}\geq 0}^{\infty }\frac{%
	a_{1}^{n_{1}}}{n_{1}!}...\frac{a_{d-1}^{n_{d-1}+s}\left( -1\right) ^{\left(
		d-1\right) s+p+l_{i-1}}}{\left( n_{d-1}+s\right) !}\left( \frac{-1}{w}\right) ^{l+r+\left( d-1\right) s}\vspace{0.15cm} \\ 
& \times f(w\left(
\left( d-1\right) s+p+l_{i-1}\right)) \frac{\left( l+r+\left( d-1\right)
	s\right) !}{\left( \left( d-1\right) s+p+l_{i-1}\right) !\left(
	l-l_{i-1}+r-p\right) !}\vspace{0.2cm} \\ 
& \times \ _{2}F_{1}\left( \left. 
\begin{tabular}{cc}
$l+r+\left( d-1\right) s,$ & $l+r+1+\left( d-1\right) s\vspace{0.2cm}$ \\ 
\multicolumn{2}{c}{$l-l_{i-1}+r+1-p$}%
\end{tabular}%
\right\vert \frac{\alpha }{w}\right) 
\end{array}%
\end{equation*}%
and, finally, for \ $p+1-r=d-1$, this means that $r=0$ and $p=d-2$, we get%
\begin{equation*}
\begin{array}{cl}
B_{d-1} & =\sum\limits_{s,k,n_{1},...,n_{d-1}\geq 0}^{\infty }\frac{%
	a_{1}^{n_{1}}}{n_{1}!}...\frac{a_{d-1}^{n_{d-1}+s+1}\left( -1\right)
	^{s+l_{i-1}}}{\left( n_{d-1}+s+1\right) !}\left( \frac{-1}{w}\right)
^{l+\left( d-1\right) \left( s+1\right) }\vspace{2mm} \\ 
& \times f(w\left( \left( d-1\right) s+d-2+l_{i-1}\right) \frac{\left(
	l+\left( d-1\right) \left( s+1\right) \right) !}{\left( s+l_{i-1}\right)
	!\left( l-l_{i-1}+1\right) !}\vspace{2mm} \\ 
& \times \ _{2}F_{1}\left( \left. 
\begin{tabular}{cc}
$l+\left( d-1\right) \left( s+1\right) ,$ & $l+1+\left( d-1\right) \left(
s+1\right) \vspace{0.2cm}$ \\ 
\multicolumn{2}{c}{$l-l_{i-1}+2$}%
\end{tabular}%
\right\vert \frac{\alpha }{w}\right) .%
\end{array}%
\end{equation*}

\section*{Acknowledgements}	
	Part of this work was performed while the  author is visiting the KU Leuven and he kindly thanks its hospitality. He also thanks the referee of this paper for his/her useful comments and suggestions.


\end{document}